\documentclass[11pt]{article}

\usepackage[letterpaper,margin=1.25in]{geometry}

\usepackage{amsmath,amssymb,amsfonts,amsthm}
\usepackage{mathtools}
\usepackage{mathrsfs}
\usepackage{appendix}
\usepackage{cite}

\usepackage{enumitem}
\usepackage{xcolor}
\usepackage[
    colorlinks=true,
    linkcolor=blue,
    citecolor=blue,
    urlcolor=blue
]{hyperref}

\numberwithin{equation}{section}
\allowdisplaybreaks

\theoremstyle{plain}
\newtheorem{theorem}{Theorem}[section]

\newtheorem{lemma}[theorem]{Lemma}

\theoremstyle{definition}

\theoremstyle{remark}
\newtheorem{remark}{Remark}

\title{Geometric refinements of Liouville-type theorems for the stationary Navier--Stokes equations in $\mathbb{R}^3$}

\author{Juhyeong Lee}

\date{}

\begin{document}

\maketitle

\begin{abstract}
We prove Liouville-type theorems for smooth solutions $(u,p)$ of the stationary Navier--Stokes equations in $\mathbb R^3$ satisfying the finite Dirichlet energy condition and the uniform decay condition. Our results give geometric refinements of two recent Osgood-type criteria, namely the relative decay criterion and the weighted integrability criterion formulated in terms of the head pressure $Q=\frac{1}{2}|u|^2+p$. The proof exploits several properties of the head pressure and the scalar triple-product structure of $(u\times\omega)\cdot\nabla Q$. Combined with an Osgood-type representation of the total vorticity energy,
this geometric structure yields Liouville-type criteria involving only the tangential interaction of $u$ and $\omega$ over the superlevel sets of $|Q|$. In the relative decay setting, they lead to Osgood-scale smallness conditions that allow subcritical growth beyond the corresponding uniform bounds. In the weighted integrability setting, the previous conditions involving the full velocity and velocity gradient are refined to involve only the tangential components of $u$ and $\omega$ along the level surfaces of $Q$, together with an angular factor measuring their relative orientation.

\medskip \noindent\textbf{Keywords:} Navier--Stokes equations; Liouville-type theorems; Osgood condition

\medskip \noindent\textbf{2020 Mathematics Subject Classification:} 35B53; 35Q30; 76D05
\end{abstract}

\section{Introduction}

In this paper, we consider the stationary Navier--Stokes equations in $\mathbb{R}^3$: 
\begin{equation} \label{eq:1.1} 
\begin{cases}
-\Delta u + (u \cdot \nabla) u +\nabla p  = 0, \\
\nabla\cdot u = 0,
\end{cases}
\end{equation}
with the uniform decay condition
\begin{equation} \label{eq:1.2}
u(x)\to 0\quad \text{as} \quad |x|\to+\infty\ .
\end{equation}
Here $u = (u_1(x), u_2(x), u_3(x))$ denotes the velocity field of the fluid and $p = p(x)$ denotes the pressure of the fluid. Throughout the paper, we denote the head pressure and the vorticity by $Q:=\frac{1}{2}|u|^2+p$ and $\omega:=\nabla\times u$, respectively. 

The Liouville problem for \eqref{eq:1.1} asks whether any smooth solution satisfying the uniform decay condition \eqref{eq:1.2} and the finite Dirichlet energy condition
\begin{equation} \label{eq:1.3} 
\int_{\mathbb R^3} |\nabla u|^2\,dx< + \infty
\end{equation}
must be trivial. This problem is formulated in Galdi's book \cite[Remark X.9.4, p.~729]{10} as the uniqueness of the zero solution in the class of stationary $D$-solutions vanishing at infinity. It is also stated as a conjecture in Tsai's book \cite[Conjecture 2.5, p.~23]{19}. Galdi proved, by means of a Caccioppoli-type inequality, that such a solution in $\mathbb R^n$ is trivial under the additional assumption $u\in L^{\frac{3n}{n-1}}(\mathbb R^n)$. For higher dimensions $n \ge 4$, this condition is covered by the Sobolev embedding and the decay of $u$ at infinity. The two-dimensional case was solved by Gilbarg and Weinberger \cite{11} using the maximum principle for the vorticity equation. Thus the only remaining open case is $n=3$.

On the other hand, several important partial results have been obtained in different directions. Chae \cite{1} obtained a Liouville-type theorem under the condition $\Delta u\in L^{6/5}(\mathbb R^3)$, a condition with the same scaling as $\nabla u \in L^{2}(\mathbb R^3)$. Chae and Wolf \cite{2} logarithmically improved Galdi's $L^{9/2}$-condition. Seregin \cite{15} proved another result under the condition $u\in L^6(\mathbb R^3)\cap BMO^{-1}(\mathbb R^3)$, where $BMO^{-1}(\mathbb R^3)$ has the same scaling as $L^3(\mathbb R^3)$. For further results, we refer to \cite{3,4,5,6,7,8,9,12,13,14,16,17,18,20} and the references therein.

A useful approach to study this problem is based on the head pressure $Q$, which plays an important role in several works \cite{1,2,4,5,7,8}. Indeed, it satisfies the elliptic equation
\[
\Delta Q-u\cdot\nabla Q=|\omega|^2.
\]
By \cite[Theorem~2.1]{1}, the assumptions \eqref{eq:1.2} and \eqref{eq:1.3} imply that $Q$ has a finite limit at infinity. After subtracting a suitable constant from the pressure, one may assume that 
\[
Q(x) \to 0 \quad \text{as} \quad |x|\to+\infty\  .
\]
Hence the maximum principle gives $Q\le0$ in $\mathbb R^3$. Again, by the maximum principle, either $Q\equiv0$, in which case the solution is already trivial, or
\[
Q(x)<0 \quad \text{for all } x\in\mathbb R^3.
\]
Thus, for any nontrivial smooth solution, the head pressure never vanishes. By Sard's theorem, for a.e. $\lambda\in(0,M)$ with $M:=\sup_{x\in\mathbb R^3}|Q(x)|$, the superlevel set $\{|Q|>\lambda\}$ has smooth boundary $\{|Q|=\lambda\}$, and thus the divergence theorem may be applied.

Recently, Chae \cite{7} further developed the head-pressure approach through level-set methods and the coarea formula, obtaining Osgood-type criteria that recover and extend previous results \cite{2,4,5}. More precisely, a continuous nondecreasing function $\xi:[0,+\infty) \to [0,+\infty)$ with $\xi(0)=0$ and $\xi(s)>0$ for $s>0$ satisfies the \emph{Osgood condition} if
\[
\int_0^1 \frac{ds}{\xi(s)}=+\infty. 
\] 
For such a function $\xi$, Chae proved that a smooth solution of \eqref{eq:1.1} satisfying \eqref{eq:1.2} and \eqref{eq:1.3} is trivial provided that, for some $q\ge1$,
\begin{equation} \label{eq:1.4}
\frac{|u|^2}{\xi(|Q|)^{\frac{3-2q}{q}}} \in L^{\frac{q}{q-1}}(\mathbb R^3), \quad
\frac{|\nabla u|}{\xi(|Q|)^{\frac{3q-3}{2q}}} \in L^{2q}(\mathbb R^3).
\end{equation}
As a corollary, Liouville-type results hold under either of the following relative bounds:
\begin{equation} \label{eq:1.5}
\sup_{x\in\mathbb R^3}\frac{|u(x)|^2}{\xi(|Q(x)|)}<+\infty
\quad\text{or}\quad
\sup_{x\in\mathbb R^3}\frac{|p(x)|}{\xi(|Q(x)|)}<+\infty .
\end{equation}

Motivated by these results, we obtain geometric refinements of the previous Osgood-type criteria \eqref{eq:1.4} and \eqref{eq:1.5}. These refinements are obtained by keeping track of the geometric structure of the scalar triple-product
\[
(u\times\omega)\cdot\nabla Q.
\]
Since $\Delta u$ is divergence-free, the Lamb form of the equation \eqref{eq:1.1}
\[
\nabla Q = \Delta u + u\times\omega
\]
yields weighted estimates on the level-set domains of $|Q|$, where we test the equation against $\nabla Q/\xi(|Q|)$. Combined with the Osgood-type representation \eqref{eq:2.1} of $\int_{\mathbb R^3}|\omega|^2\,dx$, these estimates yield Liouville-type criteria formulated in terms of the tangential components of $u$ and $\omega$ along the level surfaces of $Q$.

Before stating our main results, we introduce the tangential decomposition along the level surfaces of $Q$ and the associated angular factor. Set
\[
n_Q:=\frac{\nabla Q}{|\nabla Q|},\quad
u_T:=u-(u\cdot n_Q)n_Q,\quad
\omega_T:=\omega-(\omega\cdot n_Q)n_Q
\quad\text{on } \{|\nabla Q|>0\},
\]
and
\[
n_Q = u_T = \omega_T = 0 \quad\text{on } \{|\nabla Q|=0\}.
\] We then define the tangential angular factor $\Theta$ by
\[ 
\Theta := \frac{(u_T\times\omega_T)\cdot n_Q} {|u_T|\,|\omega_T|}
\quad \text{on } \{|u_T|\,|\omega_T|\,|\nabla Q|>0\},
\quad \Theta=0 \quad \text{on } \{|u_T|\,|\omega_T|\,|\nabla Q|=0\}.
\] 
Then $-1\le\Theta\le1$, and 
\begin{equation} \label{eq:1.6}
(u\times\omega)\cdot\nabla Q = |u_T|\,|\omega_T|\,|\nabla Q|\,\Theta 
\quad\text{in }\mathbb R^3. 
\end{equation}

We first state Osgood-scale smallness criteria.

\begin{theorem} \label{thm:1}
Let $(u,p)$ be a smooth solution of \eqref{eq:1.1} satisfying \eqref{eq:1.2} and \eqref{eq:1.3}, and let $\xi(\cdot)$ satisfy the Osgood condition. Suppose that either
\begin{equation} \label{eq:1.7}
\liminf_{\lambda \to 0}
\left(\int_{\lambda}^{1} \frac{ds}{\xi(s)}\right)^{-1}
\left\| \frac{|u_T|^2\Theta^2}{\xi(|Q|)} \right\|_{L^\infty(\{|Q|>\lambda\})} < 1,
\end{equation}
or there exists a sufficiently small constant $c>0$ such that
\begin{equation} \label{eq:1.8} 
\liminf_{\lambda \to 0}
\left(\int_{\lambda}^{1} \frac{ds}{\xi(s)}\right)^{-1}
\left\| \frac{|\omega_T|^2\Theta^2}{\xi(|Q|)} \right\|_{L^{\frac{3}{2}}(\{|Q|>\lambda\})} < c.
\end{equation}
Then $u\equiv 0$.
\end{theorem}

\begin{remark}
Since $Q < 0$ in the nontrivial case, we have
\[
p=Q-\frac12|u|^2<0, \quad |p|=|Q|+\frac12|u|^2,
\]
and hence $|u_T|^2 \le |u|^2\le 2|p|$. Thus, the velocity condition \eqref{eq:1.7} may be replaced by the
following pressure condition:
\begin{equation} \label{eq:1.9}
\liminf_{\lambda \to 0}
\left(\int_{\lambda}^{1} \frac{ds}{\xi(s)}\right)^{-1}
\left\| \frac{|p|\Theta^2}{\xi(|Q|)} \right\|_{L^\infty(\{|Q|>\lambda\})} < \frac{1}{2}.
\end{equation}
\end{remark}

\begin{remark}
We now point out that Theorem~\ref{thm:1} gives a geometric refinement of the relative decay conditions \eqref{eq:1.5}. Since the Osgood condition implies $\int_\lambda^1 \frac{ds}{\xi(s)}\to+\infty$ as $\lambda \to 0$, if the velocity condition in \eqref{eq:1.5} holds, then we have
\[
\left(\int_\lambda^1 \frac{ds}{\xi(s)}\right)^{-1}
\left\| \frac{|u_T|^2\Theta^2}{\xi(|Q|)} \right\|_{L^\infty(\{|Q|>\lambda\})}
\le
\left(\int_\lambda^1 \frac{ds}{\xi(s)}\right)^{-1}
\left\| \frac{|u|^2}{\xi(|Q|)} \right\|_{L^\infty(\mathbb R^3)}
\to 0,
\]
as $\lambda\to0$. Hence our condition \eqref{eq:1.7} is automatically satisfied. Similarly, if the pressure condition in \eqref{eq:1.5} holds, then the corresponding condition \eqref{eq:1.9} follows. Therefore, each relative decay condition in \eqref{eq:1.5} is contained in the corresponding Osgood-scale smallness condition of Theorem~\ref{thm:1}. 

The gain in the admissible growth rate can be described at the Osgood scale. The velocity condition \eqref{eq:1.5} requires
\[
\left\|
\frac{|u|^2}{\xi(|Q|)}
\right\|_{L^\infty(\{|Q|>\lambda\})}
=O(1)
\quad\text{as}\quad \lambda\to0,
\]
whereas our condition \eqref{eq:1.7} is satisfied whenever there exists a sequence $\lambda_j\to0$ such that, for some $c<1$,
\[
\left\|
\frac{|u_T|^2\Theta^2}{\xi(|Q|)}
\right\|_{L^\infty(\{|Q|>\lambda_j\})}
\le
c\int_{\lambda_j}^1\frac{ds}{\xi(s)}
\quad\text{for all sufficiently large } j.
\]
Thus, our result extends the admissible behavior from uniform boundedness to subcritical growth at the Osgood scale. The pressure case admits the analogous subcritical growth at the same scale, with coefficient $c<\frac{1}{2}$.
\end{remark}

We next state weighted integrability criteria.

\begin{theorem} \label{thm:2}
Let $(u,p)$ be a smooth solution of \eqref{eq:1.1} satisfying \eqref{eq:1.2} and \eqref{eq:1.3}, and let $\xi(\cdot)$ satisfy the Osgood condition. Suppose that
\begin{equation} \label{eq:1.10}
\frac{|u_T| |\Theta|}{\xi(|Q|)^a}\in L^q(\mathbb R^3), 
\quad \frac{|\omega_T|}{\xi(|Q|)^b}\in L^r(\mathbb R^3), 
\end{equation}
where $\frac{1}{q}+\frac{1}{r}=\frac{1}{2}$ with $2\le q,r\le +\infty$, and $a+b=\frac{1}{2}$ with $a,b\in \mathbb R$. Then $u\equiv 0$.
\end{theorem}

\begin{remark}
Theorem~\ref{thm:2} also provides a geometric refinement of the weighted integrability condition \eqref{eq:1.4}. For $s \ge 1$, the condition \eqref{eq:1.10} with the parameter choice
\[
(q,r,a,b) = \left( \frac{2s}{s-1},\,2s,\, \frac{3-2s}{2s},\, \frac{3s-3}{2s} \right)
\]
yields the assumption
\begin{equation} \label{eq:1.11} 
\frac{|u_T|^2\Theta^2}{\xi(|Q|)^{\frac{3-2s}{s}}} \in L^{\frac{s}{s-1}}(\mathbb R^3),
\quad
\frac{|\omega_T|}{\xi(|Q|)^{\frac{3s-3}{2s}}} \in L^{2s}(\mathbb R^3).
\end{equation}
Our condition \eqref{eq:1.11} improves the previous condition \eqref{eq:1.4} by involving only the tangential components along the level surfaces of $Q$. Since $0\le \Theta^2 \le 1$, $|u_T|\le |u|$, and $|\omega_T|\le|\omega|\le \sqrt{2}\,|\nabla u|$, our assumptions are weaker than the corresponding previous conditions and hence give a refinement of the weighted integrability criterion.
\end{remark}

\section{Proofs of the main theorems}

Throughout this section, we work in the nontrivial case and use the properties of the head pressure recalled in the Introduction. After normalizing the pressure, we may assume that $Q(x)\to0$ as $|x|\to+\infty$. Since 
\[
\Delta Q-u\cdot\nabla Q=|\omega|^2,
\]
the maximum principle implies that either $Q\equiv0$, in which case the solution is trivial, or $Q<0$ in $\mathbb R^3$. We henceforth consider the nontrivial case and set
\[
M:=\sup_{x\in\mathbb R^3}|Q(x)|.
\]
For \(0<\lambda<M\), define
\[
\Omega_\lambda
:=
\left\{ x\in\mathbb R^3: |Q(x)|>\lambda \right\}.
\]
Since $Q(x)\to0$ as $|x|\to+\infty$, the set $\Omega_\lambda$ is bounded. Moreover, by Sard's theorem, for a.e. $\lambda\in(0,M)$, $\lambda$ is a regular value of $|Q|$, and hence $\Omega_\lambda$ has the smooth boundary $\partial\Omega_\lambda = \left\{x\in\mathbb R^3: |Q(x)|=\lambda\right\}$. Therefore, the divergence theorem may be applied to $\Omega_\lambda$ for a.e. $\lambda\in(0,M)$.

We also let \(\xi:[0,+\infty)\to[0,+\infty)\) be a continuous nondecreasing
function such that
\[
\xi(0)=0,
\quad
\xi(s)>0 \quad \text{for }s>0,
\quad
\int_0^1\frac{ds}{\xi(s)}=+\infty.
\]

The following lemma gives the Osgood-type representation of the total energy of the vorticity established in \cite[Theorem~1.1]{7}.

\begin{lemma} \label{lem:1}
Let $(u,p)$ be a nontrivial smooth solution of \eqref{eq:1.1} satisfying \eqref{eq:1.2} and \eqref{eq:1.3}. Then the following identity holds:
\begin{equation} \label{eq:2.1}
\lim_{\lambda \to 0} 
\left\{ \left( \int_{\lambda}^{1}\frac{ds}{\xi(s)}\right)^{-1} 
\int_{\{|Q|>\lambda\}} \frac{|\nabla Q|^2}{\xi(|Q|)}\,dx \right\} 
= \int_{\mathbb R^3}|\omega|^2\,dx . 
\end{equation}
\end{lemma}

We next combine Lemma~\ref{lem:1} with the tangential decomposition of the scalar triple-product to obtain the following estimate.

\begin{lemma} \label{lem:2}
Let $(u,p)$ be a nontrivial smooth solution of \eqref{eq:1.1} satisfying \eqref{eq:1.2} and \eqref{eq:1.3}. Then the following estimate holds:
\begin{equation} \label{eq:2.2}
\int_{\mathbb R^3}|\omega|^2\,dx 
\le \liminf_{\lambda \to 0} \left(\int_{\lambda}^{1} \frac{ds}{\xi(s)}\right)^{-1}
\left\| \frac{|u_T| |\Theta|}{\xi(|Q|)^a} \right\|_{L^q(\{|Q|>\lambda\})}^2 
\left\| \frac{|\omega_T|}{\xi(|Q|)^b} \right\|_{L^r(\{|Q|>\lambda\})}^2,
\end{equation}
where $\frac{1}{q}+\frac{1}{r}=\frac{1}{2}$ with $2\le q,r\le +\infty$, and $a+b=\frac{1}{2}$ with $a,b\in \mathbb R$.
\end{lemma}

\begin{proof} [Proof of Lemma~\ref{lem:2}]
Fix a regular value $\lambda\in(0,\min\{1,M\})$. Testing the Lamb form of the equation $\nabla Q = \Delta u + u \times \omega$ by $\frac{\nabla Q}{\xi(|Q|)}$ over $\Omega_\lambda$, we have
\[
\int_{\Omega_\lambda}\frac{|\nabla Q|^2}{\xi(|Q|)}\,dx
=
\int_{\Omega_\lambda}
\frac{(u\times\omega)\cdot\nabla Q}{\xi(|Q|)}\,dx
+
\int_{\Omega_\lambda}
\frac{\Delta u\cdot\nabla Q}{\xi(|Q|)}\,dx
=: I_1+I_2 .
\]
Define the primitive
\[
F_\lambda(t):=\int_\lambda^t \frac{ds}{\xi(s)}, \quad t\ge \lambda .
\]
Since $F_\lambda(\lambda) = 0$ and $\nabla\cdot\Delta u=\Delta(\nabla\cdot u)=0$, the divergence theorem implies
\[
I_2 = -\int_{\Omega_\lambda} \Delta u\cdot \nabla F_\lambda(|Q|)\,dx 
=  -F_\lambda(\lambda)\int_{\partial\Omega_\lambda} \Delta u\cdot\nu\,dS + \int_{\Omega_\lambda} F_\lambda(|Q|)\,\nabla\cdot\Delta u\,dx =0.
\]
Therefore,
\[
\int_{\Omega_\lambda}\frac{|\nabla Q|^2}{\xi(|Q|)}\,dx
=
\int_{\Omega_\lambda}
\frac{(u\times\omega)\cdot\nabla Q}{\xi(|Q|)}\,dx.
\]
Using the identity \eqref{eq:1.6}, we obtain
\begin{equation} \label{eq:2.3}
\int_{\Omega_\lambda}\frac{|\nabla Q|^2}{\xi(|Q|)}\,dx
=
\int_{\Omega_\lambda}
\frac{|u_T|\,|\omega_T|\,|\nabla Q|\,\Theta}
{\xi(|Q|)}\,dx .
\end{equation}
Then Hölder's inequality gives
\[
\begin{aligned}
\int_{\Omega_\lambda}
\frac{|u_T|\,|\omega_T|\,|\nabla Q|\,\Theta}
{\xi(|Q|)}\,dx
&=
\int_{\Omega_\lambda}
\left(\frac{|u_T| \Theta}{\xi(|Q|)^a}\right)
\left(\frac{|\omega_T|}{\xi(|Q|)^b}\right)
\left(\frac{|\nabla Q|}{\xi(|Q|)^{1/2}}\right)\,dx  \\
&\le
\left\|
\frac{|u_T| |\Theta|}{\xi(|Q|)^a}
\right\|_{L^q(\Omega_\lambda)}
\left\|
\frac{|\omega_T|}{\xi(|Q|)^b}
\right\|_{L^r(\Omega_\lambda)}
\left(
\int_{\Omega_\lambda}
\frac{|\nabla Q|^2}{\xi(|Q|)}\,dx
\right)^{\frac{1}{2}},
\end{aligned}
\]
where $\frac{1}{q}+\frac{1}{r}=\frac{1}{2}$ with $2\le q,r\le +\infty$, and $a+b=\frac{1}{2}$ with $a,b\in \mathbb R$. 

Hence,
\[
\int_{\Omega_\lambda}
\frac{|\nabla Q|^2}{\xi(|Q|)}\,dx
\le
\left\|
\frac{|u_T| |\Theta|}{\xi(|Q|)^a}
\right\|_{L^q(\Omega_\lambda)}^2
\left\|
\frac{|\omega_T|}{\xi(|Q|)^b}
\right\|_{L^r(\Omega_\lambda)}^2 .
\]
Dividing by $\int_\lambda^1 \frac{ds}{\xi(s)}$, we have
\[
\left(\int_\lambda^1 \frac{ds}{\xi(s)}\right)^{-1}
\int_{\Omega_\lambda}\frac{|\nabla Q|^2}{\xi(|Q|)}\,dx
\le
\left(\int_\lambda^1 \frac{ds}{\xi(s)}\right)^{-1}
\left\|
\frac{|u_T| |\Theta|}{\xi(|Q|)^a}
\right\|_{L^q(\Omega_\lambda)}^2
\left\|
\frac{|\omega_T|}{\xi(|Q|)^b}
\right\|_{L^r(\Omega_\lambda)}^2.
\]
Although the preceding estimate was derived only for regular values of $\lambda$, it extends to every $\lambda\in(0,\min\{1,M\})$. Indeed, for any such $\lambda$, choose regular values $\lambda_k\downarrow\lambda$. Then
\[
\Omega_{\lambda_k}\uparrow\Omega_\lambda,
\]
and hence the corresponding integrals and $L^q$- and $L^r$-norms converge to those over $\Omega_\lambda$, with the usual interpretation when $q=+\infty$ or $r=+\infty$. Moreover,
\[
\int_{\lambda_k}^1\frac{ds}{\xi(s)}
\longrightarrow
\int_\lambda^1\frac{ds}{\xi(s)}.
\]
Thus the estimate holds for every $\lambda\in(0,\min\{1,M\})$. Taking the limit inferior as $\lambda \to 0$, Lemma~\ref{lem:1} yields
\[
\int_{\mathbb R^3}|\omega|^2\,dx
\le
\liminf_{\lambda \to 0}
\left(\int_\lambda^1 \frac{ds}{\xi(s)}\right)^{-1}
\left\|
\frac{|u_T| |\Theta|}{\xi(|Q|)^a}
\right\|_{L^q(\Omega_\lambda)}^2
\left\|
\frac{|\omega_T|}{\xi(|Q|)^b}
\right\|_{L^r(\Omega_\lambda)}^2.
\]
\end{proof}

\begin{proof} [Proof of Theorem~\ref{thm:1}]
We argue by contradiction. Suppose that $(u,p)$ is a nontrivial solution satisfying the assumptions of the theorem. Then, $\int_{\mathbb R^3}|\omega|^2\,dx>0$. 

\medskip
\noindent\textit{Case \eqref{eq:1.7}.} Assume that
\[
\alpha
:=
\liminf_{\lambda \to 0}
\left(\int_{\lambda}^{1}\frac{ds}{\xi(s)}\right)^{-1}
\left\|
\frac{|u_T|^2\Theta^2}{\xi(|Q|)}
\right\|_{L^\infty(\{|Q|>\lambda\})}
<1.
\]
Applying Lemma~\ref{lem:2} with $(q,r,a,b)=(+\infty,2,\frac{1}{2},0)$, we obtain
\[
\begin{aligned}
\int_{\mathbb R^3}|\omega|^2\,dx
&\le
\liminf_{\lambda \to 0}
\left(\int_{\lambda}^{1}\frac{ds}{\xi(s)}\right)^{-1}
\left\|
\frac{|u_T| |\Theta|}{\xi(|Q|)^{1/2}}
\right\|_{L^\infty(\{|Q|>\lambda\})}^{2}
\|\omega_T\|_{L^2(\{|Q|>\lambda\})}^{2} \\
&\le
\liminf_{\lambda \to 0}
\left(\int_{\lambda}^{1}\frac{ds}{\xi(s)}\right)^{-1}
\left\|
\frac{|u_T|^2\Theta^2}{\xi(|Q|)}
\right\|_{L^\infty(\{|Q|>\lambda\})}
\|\omega\|_{L^2(\mathbb R^3)}^{2} \\
&\le
\alpha
\int_{\mathbb R^3}|\omega|^2\,dx .
\end{aligned}
\]
Since $0\le \alpha<1$, this is a contradiction and thus $u$ must be trivial.

\medskip
\noindent\textit{Case \eqref{eq:1.8}.} 
Let $C_S>0$ be a Sobolev constant such that
\[
\|u\|_{L^6(\mathbb R^3)}
\le C_S\|\nabla u\|_{L^2(\mathbb R^3)}
= C_S\|\omega\|_{L^2(\mathbb R^3)}.
\]
Assume that for some constant $0<c\le C_S^{-2}$,
\[
\beta := \liminf_{\lambda \to 0}
\left(\int_{\lambda}^{1} \frac{ds}{\xi(s)}\right)^{-1}
\left\| \frac{|\omega_T|^2\Theta^2}{\xi(|Q|)} \right\|_{L^{\frac{3}{2}}(\{|Q|>\lambda\})} < c.
\]
Starting from \eqref{eq:2.3} and repeating the proof of Lemma~\ref{lem:2} with the angular factor $\Theta$ assigned to the vorticity term, we have
\[
\int_{\mathbb R^3}|\omega|^2\,dx 
\le \liminf_{\lambda \to 0} \left(\int_{\lambda}^{1} \frac{ds}{\xi(s)}\right)^{-1}
\left\| \frac{|u_T|}{\xi(|Q|)^a} \right\|_{L^q(\{|Q|>\lambda\})}^2 
\left\| \frac{|\omega_T||\Theta|}{\xi(|Q|)^b} \right\|_{L^r(\{|Q|>\lambda\})}^2.
\]
Choosing $(q,r,a,b)=(6,3,0,\frac{1}{2})$ and using the Sobolev embedding, we obtain
\[
\begin{aligned}
\int_{\mathbb R^3}|\omega|^2\,dx
&\le
\liminf_{\lambda \to 0}
\left(\int_\lambda^1\frac{ds}{\xi(s)}\right)^{-1}
\left\|
\frac{|\omega_T||\Theta|}{\xi(|Q|)^{1/2}}
\right\|_{L^3(\{|Q|>\lambda\})}^2
\|u_T\|_{L^6(\{|Q|>\lambda\})}^2
\\
&\le
\liminf_{\lambda \to 0}
\left(\int_\lambda^1\frac{ds}{\xi(s)}\right)^{-1}
\left\|
\frac{|\omega_T|^2\Theta^2}{\xi(|Q|)}
\right\|_{L^{3/2}(\{|Q|>\lambda\})}
\|u\|_{L^6(\mathbb R^3)}^2
\\
&\le
\liminf_{\lambda \to 0}
\left(\int_\lambda^1\frac{ds}{\xi(s)}\right)^{-1}
\left\|
\frac{|\omega_T|^2\Theta^2}{\xi(|Q|)}
\right\|_{L^{3/2}(\{|Q|>\lambda\})}
C_S^2\|\omega\|_{L^2(\mathbb R^3)}^2
\\
&\le
C_S^2\beta
\int_{\mathbb R^3}|\omega|^2\,dx.
\end{aligned}
\]
Since $C_S^2\beta <1$, this is a contradiction and thus $u \equiv 0$.

\end{proof}

\begin{proof} [Proof of Theorem~\ref{thm:2}]
We also argue by contradiction. Suppose that $(u,p)$ is a nontrivial solution satisfying the assumptions of the theorem. Then, $\int_{\mathbb R^3}|\omega|^2\,dx>0$. Assume that
\[
\frac{|u_T| |\Theta|}{\xi(|Q|)^a} \in L^q(\mathbb R^3),
\quad \frac{|\omega_T|}{\xi(|Q|)^b}\in L^r(\mathbb R^3),
\]
where $\frac{1}{q}+\frac{1}{r}=\frac{1}{2}$ with $2\le q,r\le +\infty$, and $a+b=\frac{1}{2}$ with $a,b\in \mathbb R$.

Since $\lim_{\lambda\to0}\left(\int_{\lambda}^{1} \frac{ds}{\xi(s)}\right)^{-1} = 0$, by Lemma~\ref{lem:2},
\[
\begin{aligned}
\int_{\mathbb R^3}|\omega|^2\,dx
&\le
\liminf_{\lambda \to 0} \left(\int_{\lambda}^{1} \frac{ds}{\xi(s)}\right)^{-1}
\left\| \frac{|u_T| |\Theta|}{\xi(|Q|)^a} \right\|_{L^q(\{|Q|>\lambda\})}^2 
\left\| \frac{|\omega_T|}{\xi(|Q|)^b} \right\|_{L^r(\{|Q|>\lambda\})}^2  \\
&\le
\liminf_{\lambda \to 0}
\left(\int_{\lambda}^{1} \frac{ds}{\xi(s)}\right)^{-1}
\left\| \frac{|u_T| |\Theta|}{\xi(|Q|)^a} \right\|_{L^q(\mathbb R^3)}^2
\left\| \frac{|\omega_T|}{\xi(|Q|)^b} \right\|_{L^r(\mathbb R^3)}^2 = 0.
\end{aligned}
\]
Therefore, this is a contradiction and thus $u \equiv 0$.
\end{proof}

\section*{Declarations}

\subsection*{Funding}

No funding was received for conducting this study.

\subsection*{Competing Interests}

The author has no relevant financial or non-financial interests to disclose.

% ============================================================
% Author information
% ============================================================

\bigskip

\medskip

\noindent\textbf{Juhyeong Lee.}
Department of Mathematics, University of British Columbia, Vancouver, BC V6T 1Z2, Canada;
email: \texttt{juhyeong.lee.math@gmail.com}

\end{document}